\documentclass{amsart}

\usepackage{mathtools, 
            amssymb, 
            hyperref, 
            graphicx,
}

\makeatletter
\renewcommand{\sectionautorefname}{\S\@gobble}
\renewcommand{\subsectionautorefname}{\S\@gobble}
\renewcommand{\subsubsectionautorefname}{\S\@gobble}
\makeatother

\usepackage[backend=biber, url=false, doi=false, isbn=false, maxnames=10, maxalphanames=5]{biblatex}
\usepackage{amsthm}
\usepackage{thmtools}
\theoremstyle{plain}
\newtheorem{theorem}{Theorem}[section]
\newtheorem{lemma}[theorem]{Lemma}

\newtheorem{corollary}[theorem]{Corollary}

\theoremstyle{definition}

\newtheorem{example}[theorem]{Example}

\theoremstyle{remark}

\newcommand{\paren}[1]{\left( #1 \right)}

\usepackage[dvipsnames]{xcolor}
\newcommand{\changed}[1]{{\textcolor{black}{#1}}}

\title[]{Obtaining the Chamanara Surface from \\ the van der Corput sequence}
\author[]{Zawad Chowdhury, Francois Clement, Max Horwitz}

\begin{document}
\begin{abstract}
    We investigate a family of $4$-regular graphs constructed to test for the presence of combinatorial structure in a sequence of distinct real numbers. 
    We show that the graphs constructed from the Kronecker sequence can be embedded into the torus, while the graphs constructed from the binary van der Corput sequence can be embedded into the Chamanara surface, in both cases with the possible removal of one edge. 
    \changed{These results generalize to embeddings of sequence graphs coming from interval exchange transformations into associated translation surfaces.}
\end{abstract}
\maketitle

\section{Discussion}

\subsection{Introduction} Let $(a_n)_{n \in \mathbb{Z}_{\ge 0}}$ be a sequence of distinct real numbers. From the first $N$ terms $a_0, \dots, a_{N-1}$ of this sequence, we construct a 4-regular graph on the vertices $\{0, 1, \dots, N-1\}$ as follows. First we connect the edges $(0, 1)$, then $(1, 2)$, and so on until $(N-1, 0)$. Then we find the permutation $\pi: \{0, 1, \dots, N-1\} \to \{0, 1, \dots, N-1\}$ which orders the terms $a_0, \dots, a_{N-1}$ so that 
$$a_{\pi(0)} < a_{\pi(1)} < \dots < a_{\pi(N-1)},$$
and we connect the edges $(\pi(0), \pi(1))$, $(\pi(1), \pi(2))$ and so on until $(\pi(N-1), \pi(0))$. 
This yields a $4$-regular graph whose edge set $E$ is composed of two Hamiltonian cycles. The first visits vertices in the order defined by the sequence, and the second in the order of their size. We define this graph to be the \textbf{$N$-th sequence graph} associated with $(a_n)_{n \in \mathbb{Z}_{\ge 0}}.$ 
\begin{figure}[!ht]
    \centering
    \begin{minipage}{0.48\textwidth}
        \centering
        \includegraphics[width=0.8\textwidth]{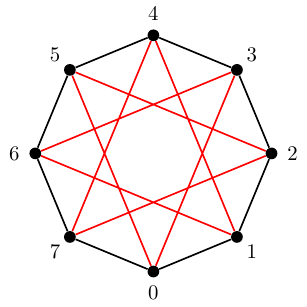}
    \end{minipage}
        \begin{minipage}{0.48\textwidth}
        \centering
        \includegraphics[width=0.8\textwidth]{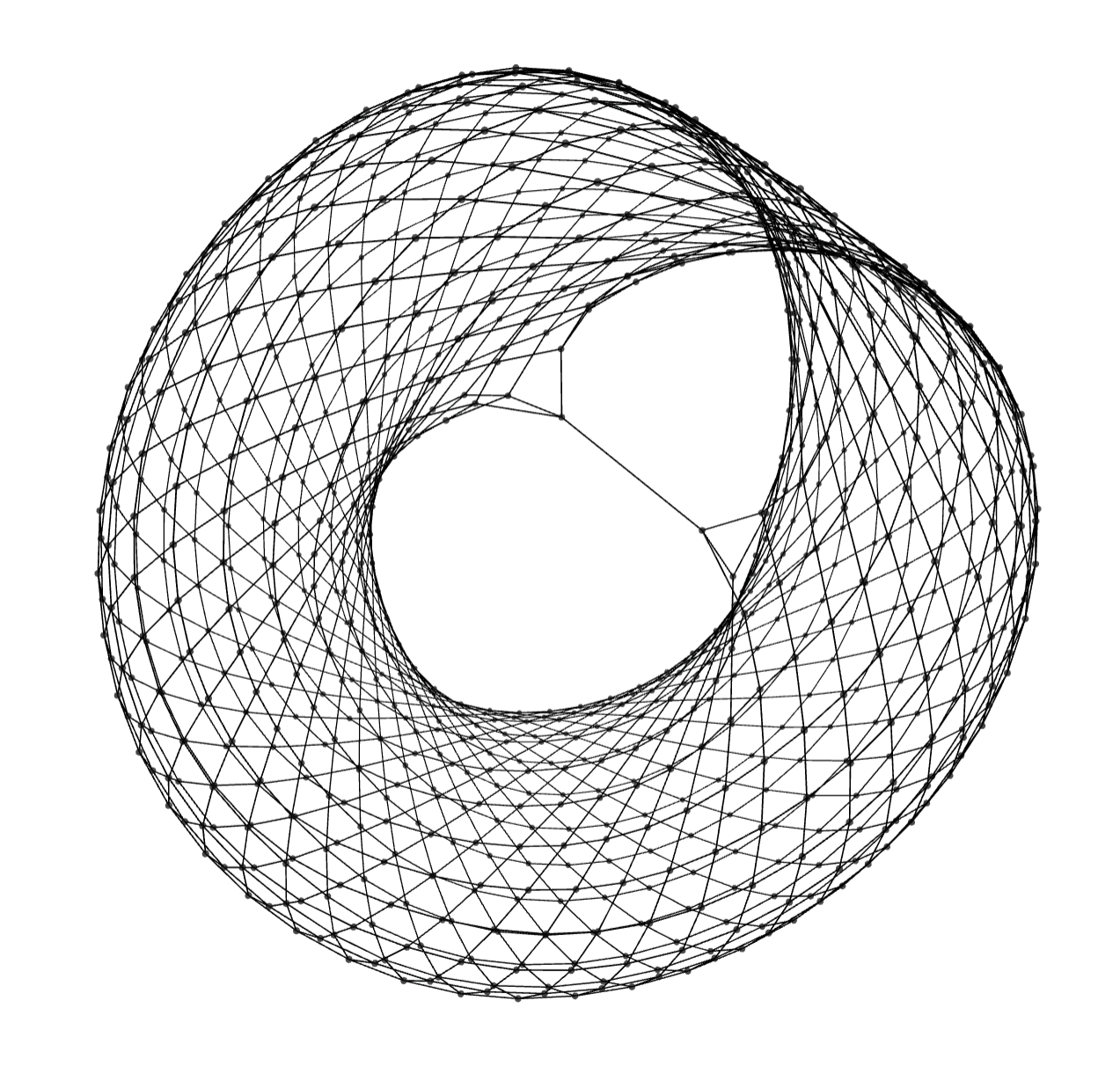}
    \end{minipage}
    \caption{Left: Kronecker sequence graph with $N=8$; colors show the two Hamiltonian cycles. Right: Kronecker sequence graph with $N = 971$ resembles a torus besides one rogue edge.}
    \label{fig:kron-seq-graphs}
\end{figure}
\begin{example} \label{ex:kron-8-graph}
    Consider the Kronecker sequence with parameter $\theta = (1+\sqrt{5})/2 = 1.618\dots$, whose terms are given by $a_n = n\theta \mod 1$ \changed{(see \autoref{sec:kron-intro} below)}. Starting with $n=0$, the first 8 terms of this sequence are roughly 0, 0.618, 0.236, 0.854, 0.472, 0.090, 0.708, 0.326.
    They are sorted by the permutation $(0, 5, 2, 7, 4, 1, 6, 3)$, where position $i$ indicates the value of $\pi(i)$. The corresponding sequence graph is illustrated in \autoref{fig:kron-seq-graphs}.
\end{example}

Sequence graphs were first defined by Steinerberger \cite{steinerberger2020expandertest} as a test to determine whether a list of numbers resembles independent samples of a random variable. Korssjoen, Li, Steinerberger, Tripathi and Zhang \cite{korssjoen2022structures} investigated them further, catalogued certain structures found in many deterministic sequence graphs. In this paper, we focus on the graphs for two well-known uniformly distributed sequences, Kronecker and van der Corput, and fully explain the structures present.

We say a graph $G$ embeds into a surface $S$ if we can draw the vertices and edges of $G$ in $S$ without any crossings. In \autoref{sec:kron-intro} we describe how the Kronecker sequence graphs embed into the torus. We discuss the van der Corput case in \autoref{sec:vdc-intro}, defining the Chamanara surface \cite{chamanara2004infinite} and describing how the sequence graphs embed into it. Finally, we explore a possible explanation for the surface structure in both sequences in \autoref{sec:iet-intro}. Proofs of our theorems are deferred till \autoref{sec:proofs}.

\subsection{Kronecker} \label{sec:kron-intro}
The \textbf{Kronecker} sequence $(a_n)_{n \in \mathbb{Z}_{\ge 0}}$ is the sequence given by $$a_{n} = n \theta \mod1 $$
where $\theta \in \mathbb{R} \setminus \mathbb{Q}$. The fact that $\theta$ is irrational ensures that the elements of the sequence are pairwise distinct. Kronecker sequence graphs are illustrated in \autoref{fig:kron-seq-graphs}.

Since $a_0 = 0$ is the lowest term, we will always have $\pi(0) = 0$ regardless of $N$. Besides that, the minimum and maximum value of $a_i$ in our first $N$ terms, $\pi(1)$ and $\pi(N-1)$, are important for our result. Whenever $N = \pi(1) + \pi(N-1)$, then the gaps $\pi(i+1) - \pi(i)$ in our permutation take on one of two values. Moreover these two values are congruent mod $N$, which is enough to show that the sequence graphs embed into the torus, our first result. 

\begin{theorem}[label=thm:kron-nice, restate=thmkronnice]
For $N = \pi(1) + \pi(N-1)$, the $N$-th Kronecker sequence graph can be embedded into a torus. 
\end{theorem}
 The detailed proof is given in \autoref{sec:kron-proof}. The condition $N = \pi(1) + \pi(N-1)$ does not occur very often; \changed{it corresponds to the denominators of the best rational approximations of $\theta$}. It can be shown that $N$ for which this holds grow exponentially, based on the continued fraction expansion of $\theta$. \changed{For example, if $\theta$ is the golden ratio $\frac{1 + \sqrt{5}}{2}$ then this property holds whenever $N$ is a Fibonacci number, while for $\theta=\sqrt{2}$ it would be whenever $N$ is a Pell number.} For $N \ne \pi(1) + \pi(N-1)$, our graph can be embedded into the torus, up to the deletion of a single edge.

\begin{theorem}[label=thm:kron-hard, restate=thmkronhard]
For $N \neq \pi(1) + \pi(N-1)$, the $N$-th Kronecker sequence graph can be embedded into a torus after deleting the single edge $(N-1, 0)$. 
\end{theorem}

The above theorem follows \changed{from \autoref{thm:kron-nice} and} a more general result that applies to any sequence graph; the detailed proof is in \autoref{sec:top-subgraphs}.
\changed{The embedding in \autoref{thm:kron-nice} is isometric, and therefore the graph tesselates the torus with squares. So as $N$ tends to infinity, the Kronecker sequence graphs produce better and better approximations of the torus.} Our proof for \autoref{thm:kron-hard} through the general result loses this information, but the produced embeddings still approximate the torus well, indicating that structure is mostly preserved.

\subsection{Van der Corput} \label{sec:vdc-intro}
The van der Corput sequence~\cite{vdc} is a famous sequence which is well distributed in the unit interval. \changed{It was introduced at the beginning of the study of ``uniform distributions'' (i.e., how well does a point set approximate the continuous distribution), and was later shown to have the optimal asymptotic order of regularity~\cite{vAE}.} There is a $b$-ary van der Corput sequence for any base $b$; we focus on the binary case. This is arguably the most common as it is the most uniformly distributed in $[0,1]$. 

The (binary) \textbf{van der Corput} sequence $(a_n)_{n \in \mathbb{Z}_{\ge 0}}$ is the sequence where $a_n$ is the number obtained by reversing the binary expansion of $n$ and then placing it after the decimal point. To be precise, if $d_k(n)$ is the $k$-th binary digit of $n$ (i.e., $n = \sum_{k=0}^{L-1} d_k(n)2^k$), then we can write $a_n$ as $\sum_{k=0}^{L-1} d_k(n)/2^{k+1}$.

\begin{figure}[!ht]
    \centering
    \begin{minipage}{0.48\textwidth}
        \centering
        \includegraphics[width=0.7\textwidth]{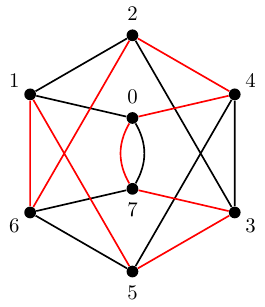}
    \end{minipage}
        \begin{minipage}{0.48\textwidth}
        \centering
        \includegraphics[width=0.8\textwidth]{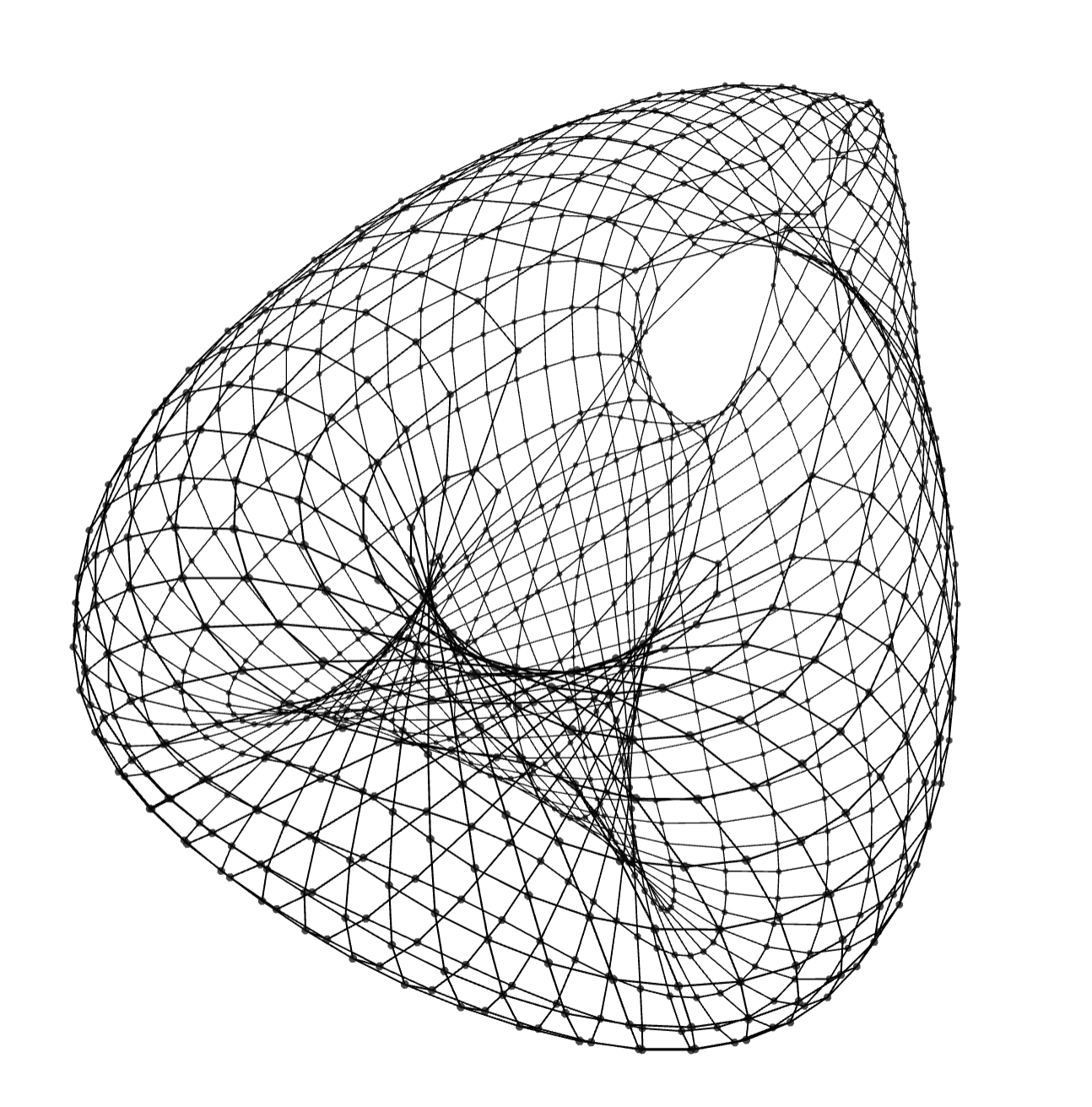}
    \end{minipage}
    \caption{Left: van der Corput sequence graph at $N = 8$. Right: van der Corput sequence graph at $N = 1024$ resembles a surface.}
    \label{fig:vdc-seq-graphs}
\end{figure}
\begin{example} \label{ex:vdc-8-graph}
    The first 8 nonnegative integers in binary are 
    $0, 1, 10, 11, 100, 101$, $110, 111$.
    By reversing and placing them after the decimal point, we obtain the first 8 terms of the binary van der Corput sequence $0.0, 0.1, 0.01, 0.11, 0.001, 0.101$, $0.011$, $0.111$.
    In decimal, these terms are $0, 0.5, 0.25, 0.75, 0.125, 0.625, 0.375, 0.875$, which are sorted by the permutation $(0, 4, 2, 6, 1, 5, 3, 7)$. The corresponding sequence graph is illustrated in \autoref{fig:vdc-seq-graphs}.
\end{example}

The binary van der Corput sequence graphs embed into a surface that is a bit more exotic than the torus, called the \textbf{Chamanara surface} \cite{, artigiani2025veechcovers,chamanara2004infinite}. 
It is illustrated in \autoref{fig:chamanara-surf}. 
We shall define the surface as a quotient space of the square, by identifying segments on the sides.
To construct it, consider a square with sides of length 1. We divide the top and bottom side both into two halves. Then we glue the left half on top with the right half on bottom; call this segment $h_1$. Next we divide the remaining halves, and again glue the top left part with the bottom right part from that division (calling this segment $h_2$). We repeat this process ad infinitum, obtaining segments $h_3, h_4, \dots$ and so on. We do the same with the left and right sides, identifying the top part of the left side with the bottom part of the right side. This produces segments $v_1, v_2, \dots$ and so on. The resulting surface is the {Chamanara surface}. 
\begin{figure}[!ht]
    \centering
    \includegraphics[width=0.5\textwidth]{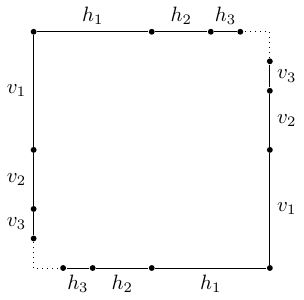}
    \caption{The Chamanara Surface}
    \label{fig:chamanara-surf}
\end{figure}

The Chamanara surface is a connected two-dimensional manifold. Since all of its gluings identify parallel segments of equal length, there is a maximal atlas on the surface with all transition functions given by translations. Thus the Chamanara surface is an example of a \emph{translation surface}. \changed{However it is not compact, has infinite genus and a wild singularity \cite{delecroix2024infinite}.} Therefore it does not fall under the regime of \emph{finite} translation surfaces, and is instead an example of a \emph{Loch Ness Monster}. 

Chamanara \cite{chamanara2004infinite} introduced a one parameter family of surfaces $X_\alpha$ for $\alpha \in (0, 1)$ as constructions of surfaces with non-elementary affine automorphism groups; the Chamanara surface we define is $X_\alpha$ with $\alpha=1/2$. 
Since the original definition, this surface has been a key example of a translation surface with large Veech group; see the notes by Herrlich and Randecker \cite{herrlich2016veechchamanara} for a review of the calculation, or the introduction of \cite{randecker2018wild} for an overview of the literature.
Recently, Artigiani, Randecker, Sadanand, Valdez and Weitze-Schmith\"{u}sen \cite{artigiani2025veechcovers} used covers of the Chamanara surface to produce novel constructions of translation surfaces whose Veech groups are the free groups.

This paper finds the Chamanara surface in a different context from the existing literature, as the surface which van der Corput sequence graphs embed into.
\begin{theorem}[label=thm:vdc-nice, restate=thmvdcnice]
    For $N = 4^m$, the $N$-th binary van der Corput sequence graph can be embedded into the Chamanara surface.
\end{theorem}
The proof is given in \autoref{sec:vdc-proof}. Our proof can be extended to show that the $b$-ary van der Corput sequence embeds into a Chamanara surface $X_\alpha$ with $\alpha = 1/b$.
Similar to the Kronecker case, we can embed sequence graphs without the perfect number of vertices into the surface by removing a single edge.

\begin{theorem}[label=thm:vdc-hard, restate=thmvdchard]
    For $N \ne 4^m$, the $N$-th binary van der Corput sequence graph can be embedded into the Chamanara surface after deleting the single edge $(N-1, 0)$.
\end{theorem}

\changed{For any finite $N$, our embedding does not require an infinite type surface. However, as $N$ tends to infinity, the translation surfaces fitting our embeddings better and better approximate the Chamanara surface, and therefore it is the natural surface that van der Corput sequence graphs embed into for all values of $N$. An alternative explanation, using the connections between the van der Corput sequence, Chamanara surface and the dyadic odometer, is described in \autoref{sec:iet-proof}.}

One might be dissatisfied by these results as an explanation for the pictures of van der Corput sequence graphs with high $N$ (\autoref{fig:vdc-seq-graphs}), which look more like a genus two surface than a Loch Ness Monster. The main reason is that the picture is obtained by an embedding which minimizes crossings according to an imperfect heuristic. As a result, there are $O(\log N)$ edges in the picture with crossings. These fine-grain structures show regions where the high level approximation differs from the true shape of the Chamanara surface (or its finite approximations).

\subsection{Sequence graphs from interval exchange transforms} \label{sec:iet-intro}

We found similar results in both the Kronecker and van der Corput sequences, with a precise embedding for some rare well-chosen values of $N$, and an embedding up to one edge, always $(N-1,0)$, for all other values of $N$. \changed{There is a more general phenomenon that explains these similarities, which we briefly allude to in this section and fully describe at the end of the paper in \autoref{sec:iet-proof}. While this general result could be used or adapted to prove all the previous results in this paper, we have chosen to keep the simpler proofs for ease of understanding.}  

A key fact about the Kronecker sequence, used in our proof of \autoref{thm:kron-nice}, is that there are only two possible differences $a_{i+1} - a_{i}$. So $a_{i+1}$ in this case is a piecewise linear function of $a_i$. The van der Corput sequence has a similar structure, though it is more subtle and not explicitly spelled out in our proof of \autoref{thm:vdc-nice}. In that case, among the first $2^m$ terms there are $m$ possible differences $a_{i+1} - a_i$.
Then $a_{i+1}$ is a piecewise linear function of $a_i$, with infinitely many pieces which are also exponentially decaying in size. 

\begin{figure}[!ht]
    \centering
    \begin{minipage}{0.48\textwidth}
        \centering
        \includegraphics[width=0.75\textwidth]{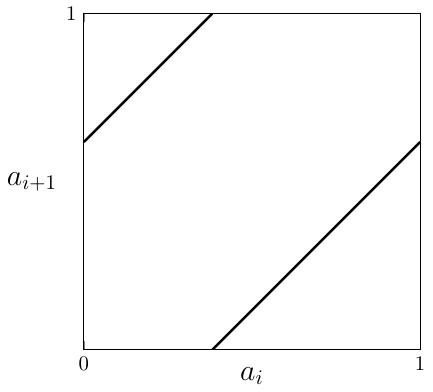}
    \end{minipage}
    \begin{minipage}{0.48\textwidth}
        \centering
        \includegraphics[width=0.75\textwidth]{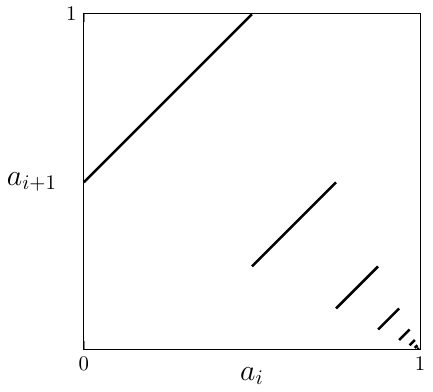}
    \end{minipage}
    \caption{Plots of $(a_i, a_{i+1})$ as $i$ varies for Kronecker sequence on the left and van der Corput on the right.}
    \label{fig:vdc-kron-iets}
\end{figure}

\changed{For both the Kronecker and van der Corput sequences the linear functions on the pieces are translations, \changed{implying that the overall piecewise linear functions are} \emph{interval exchange transformations}. In this context, our embeddings are part of a general phenomenon. Whenever a sequence is defined by an interval exchange transform, there is an associated \emph{translation surface} on which the sequence graph has an embedding. We introduce the terminology to make this statement precise (\autoref{thm:iet-embedding-thm}) and provide a proof in \autoref{sec:iet-proof}.}

\section{Proofs} \label{sec:proofs}
We first establish some notation shared by all proofs in \autoref{sec:directions-proof}. The embedding of the Kronecker sequence graph (\autoref{thm:kron-nice}) is proved in \autoref{sec:kron-proof}, and the embedding of the van der Corput sequence graph (\autoref{thm:vdc-nice}) is proved in \autoref{sec:vdc-proof}. Then we prove a result in \autoref{sec:top-subgraphs} from which both \autoref{thm:kron-hard} and \autoref{thm:vdc-hard} follow. Finally, \autoref{thm:iet-embedding-thm} is introduced and proved in \autoref{sec:iet-proof}.

\subsection{Directions in the sequence graph} \label{sec:directions-proof}
The $N$-th sequence graph contains edges of two types: those connecting vertices in the order defined by the sequence, and those connecting vertices in sorted order. These two types define two Hamiltonian cycles which, if our sequence graph approximates a surface, should behave like the two orthogonal directions in the tangent space. The edges connecting vertices in sequence order are of the form $(i, i+1)$, except the very last one being $(N-1, 0)$; we label this Hamiltonian cycle $C_1$. The edges connecting vertices in sorted order are of the form $(\pi(i), \pi(i+1))$, except the very last one being $(\pi(N-1), \pi(0))$; we label this Hamiltonian cycle $C_\pi$.

It is useful to define a successor function $S: \{0, 1, \dots, N-1\} \to \{0, 1, \dots, N-1\}$ which identifies the vertex that follows any given vertex $i$ in the cycle $C_\pi$. This function is given by
\[
    S(i) = \pi\paren{\pi^{-1}(i) + 1 \mod N}.
\]
$S$ is a bijection and thus has a well-defined inverse. For any vertex $i$ in the $N$-th sequence graph, the four edges incident to it connect it to $i+1, i-1, S(i), S^{-1}(i)$, with the first two connecting along edges in $C_1$ and the last two connecting by edges in $C_\pi$.

\subsection{\changed{The Kronecker embedding}} \label{sec:kron-proof}

Graphs where the vertex set is $\{0, 1, \ldots, N-1\}$ and the edge set connects each vertex $v$ to $v \pm c_i \mod N$ for constants $c_1, c_2, \ldots, c_k$ are called \textbf{circulant graphs}, and denoted $C_N(\{c_1, c_2, \dots, c_k\})$. Circulant graphs have very well understood properties; in particular, Costa, Strapasson, Alves and Carlos \cite{COSTA2010369} have shown certain circulant graphs embed into the torus. 

\begin{lemma}\label{lem:circulant-torus}\cite[Proposition 4]{COSTA2010369} Let $C_N(\{c_1, c_2\})$ be a connected circulant graph. Then $C_N(\{c_1 , c_2\})$ can be embedded (isometrically) into the torus. 
\end{lemma}

Ultimately, we will demonstrate that the Kronecker sequence graph is a circulant graph of this type. The fact that this embedding is isometric (every embedded edge has equal length) will demonstrate that not only does the Kronecker sequence graph embed in the torus, but for larger and larger values of $N$, they will better and better approximate the surface. 

Recall that the edges of the sequence graph lie in two cycles $C_1$ and $C_\pi$, and that the successor function $S(i)$ outputs the adjacent vertex to $i$ in $C_\pi$ (i.e., the next vertex in the sorted order defined by $\pi$). For the Kronecker sequence, the permutation $\pi$ has a nice structure described by the \emph{Three Gap Theorem}. This theorem was originally proved by S\'os \cite{Sos}; our proof uses the formulation by van Ravenstein \cite{van_Ravenstein_1988} which describes the gaps $S(i) - i$ in terms of the permutation $\pi$. 

\begin{lemma}\label{lem:three-gap}(Three Gap Theorem, \cite[Theorem 2.2]{van_Ravenstein_1988}) Let $\pi : \{0, 1, \dots, N-1\} \to \{0, 1, \dots, N-1\}$ be the permutation which sorts the first $N$ terms of the Kronecker sequence. Let $S(i) = \pi^{-1}(\pi(i) + 1)$. Then we have that
    $$S(i) - i = \begin{cases}
\pi (1) &  0 \le i < N - \pi(1), \\
-\pi(N-1) & \pi(N-1)  \le i < N,\\
\pi (1) - \pi(N-1) & N - \pi(1) \le i < \pi(N-1). \\
    \end{cases}
$$
\end{lemma}

When $N = \pi(1) + \pi(N-1)$, no values of $i$ satisfy the last case of \autoref{lem:three-gap}, and so there are only two possible values for the gap $S(i) - i$. Moreover in this case $\pi(1) \equiv -\pi(N-1) \mod  N$, which means the gaps are equivalent modulo $N$. \changed{As mentioned previously, this condition corresponds to the number of points being the denominator of a partial convergent of the irrational used for the Kronecker sequence. For a more detailed explanation, see the original paper of van Ravenstein~\cite{van_Ravenstein_1988}, and in particular Theorem 3.3.}

\begin{corollary}\label{cor:two-gap} When $N = \pi(1) + \pi(N-1)$ we have that
    $$S(i)-i \equiv \pi(1) \mod N.$$
\end{corollary}

With this, we have everything we need to prove the theorem.

\begin{proof}[Proof of \autoref{thm:kron-nice}]
Recall that the $N$-th sequence graph has an edge set composed of two Hamiltonian cycles. The first, $C_1$, connects any two vertices that differ by $1$ modulo $N$. The second, $C_\pi$, connects $i$ to $S(i)$ for every $i$. By \autoref{cor:two-gap}, this is equivalent to connecting every pair of vertices that differ by $\pi(1)$ modulo $N$. Thus we have that $$G_N \cong C_N (\{1, \pi(1)\}).$$
By \autoref{lem:circulant-torus}, we conclude that these graphs embed into the torus. 
\end{proof}

\subsection{The van der Corput embedding} \label{sec:vdc-proof}
\changed{We shall describe the embedding of the van der Corput sequence graphs into the Chamanara surface for $N = 4^m$. The same ideas can produce an embedding for $N=2^m$, but would require a non-square presentation of the Chamanara surface and a slightly more complicated proof.}

\changed{We shall represent the vertices as binary strings rather than integers.}
For $N=4^m$, define the function $b: \{0, \dots, 4^m-1\} \to \{0, 1\}^{2m}$ as the operation sending $i$ to its binary representation. We also introduce the function $r: \{0, 1\}^{2m} \to \{0, 1\}^{2m}$ which reverses the binary string. For example, if $m=2$ we have $ b(1) = 0001$ and $b(10) = 1010$, while $r(b(1)) = 1000$ and $r(b(10)) = 0101$.
The function $b$ is a bijection from the vertices of the $N$-th binary van der Corput sequence graph and length $2m$ binary strings, and $r$ is also a bijection on length $2m$ binary strings.

\begin{lemma}[label=lem:vdc-succ]
    In the $4^m$-th binary van der Corput sequence graph, for any $i$ the binary representation of the successor $S(i)$ is given by $b( S (i)) = r(r\circ b(i) + 1)$. 
\end{lemma}
\changed{This lemma shows that the binary representation of the successor function $S$ is equivalent to taking the binary representation $b(i)$, reversing it, adding 1 and then reversing back again (for a final result of $r(r\circ b(i) + 1)$). Therefore the binary successor function $b(S(i))$ is the same as binary addition ``from the left''.}
\begin{proof}
    The first $4^m$ terms of the van der Corput sequence consist of the fractions $k/4^m$ for $k \in \{0, 1, \dots, 4^m-1$\}. These fractions are ordered by their numerator, and thus we have $\pi(i) = 4^ma_i$.

    The van der Corput sequence is defined by reversing the binary representation as $a_i = b^{-1}\circ r \circ b(i) / 4^m$. Therefore we have $\pi(i) = b^{-1}\circ r \circ b(i)$. Taking the inverse and using the fact that $r^{-1} = r$, we get $\pi^{-1}(i) = b^{-1}\circ r \circ b(i)$. Therefore we can calculate the successor function
    \[
        S(i) = \pi\paren{\pi^{-1}(i) + 1 \mod 4^m} = b^{-1}\circ r \circ b\paren{b^{-1}\circ r \circ b(i) + 1 \mod 4^m}
    \]
    We can treat adding 1 to an integer the same as adding 1 to the binary representation with $2m$ bits. Thus we can conclude that $b( S (i)) = r(r\circ b(i) + 1)$.
\end{proof}

We are now ready to describe the embedding of the $4^m$-th van der Corput sequence graph in the Chamanara surface. Our embedding splits the binary representation of the vertices into two parts. For a given $m$, let $b_0, b_1: \{0, 1, \dots, 4^m-1\} \to \{0, 1\}^m$ be the functions such that, from the left, $b_0(i)$ are the first $m$ bits of $b(i)$ and $b_1(i)$ are the last $m$. For example, with $m = 2$, we have $b(1) = 0001$ and $b_0(1) = 00, b_1(1) = 01$, or $b(10) = 1010$ and $b_0(10) = b_1(10) = 10$.

\begin{proof}[Proof of \autoref{thm:vdc-nice}]
    We first describe concretely the Chamanara surface $S$ that we shall embed the binary van der Corput sequence graph into. The surface $S$ will be obtained from the unit square scaled by $\sqrt{N} = 2^m$. We also translate the square in both coordinates by $\delta = -1/2 - \varepsilon$, where $\varepsilon$ is a positive number smaller than $1/4\cdot2^{m}$. This small shift by $\delta$ (and choice of small enough $\varepsilon$) ensures that none of the integer grid lines pass through a singularity of the Chamanara surface. Therefore the surface $S$ that we work with is given by the square with corners $(\delta, \delta)$, $(\delta, 2^{m} + \delta)$, $(2^{m} + \delta, \delta)$ and $(2^{m} + \delta, 2^{m} + \delta)$, and segment identifications defined by the halving procedure described in \autoref{sec:vdc-intro}.

    Suppose $G_N$ is the $4^m$-th van der Corput sequence graph. To prove the theorem, we describe an embedding $\psi: G_N \to S$ by where it sends the vertices and the edges. There are $2^m \times 2^m$ lattice points inside this square with integer coordinates, which we treat as length $m$ binary strings. We can embed the $N = 4^m$ vertices of $G_N$ into these lattice points, so that $\psi(i) = (b_1(i), r\circ b_0(i))$. 
    \begin{figure}[!ht]
        \centering
        \includegraphics[width=0.6\textwidth]{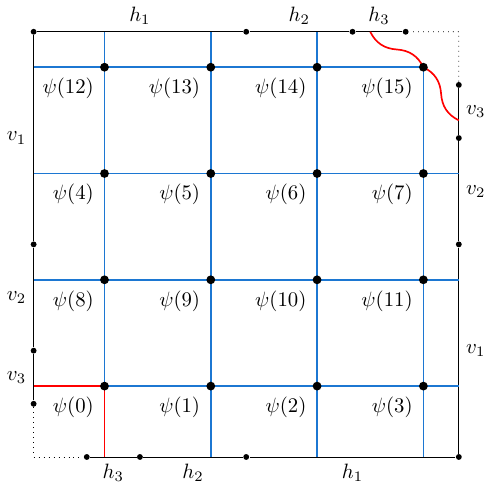}
        \caption{Embedding $\psi$ of the 16th van der Corput sequence graph into the Chamanara Surface. Blue edges follow integer grid lines (cases 1-4), while red edges are rerouted (case 5).}
        \label{fig:vdc-embedding}
    \end{figure}
    
    Since the vertices are embedded at lattice points, the natural place to embed edges would be along integer grid lines, which only cross at lattice points. If we could do this for all the edges, we would have an embedding without crossings. In fact, our embedding $\psi$ will embed all but two edges along grid lines. The two remaining edges both connect $0$ to $4^m-1$ and are embedded with a slight reroute that does not introduce any crossings. We describe our edge embeddings in cases.

    \textbf{Case 1 (Vertical inner edges):} Suppose our edge $e$ is in $C_\pi$, which means $e = (i, S(i))$, and that $b_0(i)$ is not all-ones. Then $b(i)$ has a zero in the first $m$ bits, so adding from the left as described in \autoref{lem:vdc-succ} leaves the last $m$ bits unchanged. Therefore we have $b_1(i) = b_1(S(i))$ and $r \circ b_0(S(i)) = r \circ b_0(i) + 1$. Thus if $\psi(i) = (x, y)$, we have $\psi(S(i)) = (x, y+1)$. This produces an embedding $\psi(e)$ along a vertical grid line between two lattice points.

    \textbf{Case 2 (Horizontal inner edges):} Suppose the edge $e$ is in $C_1$, which means $e = (i, i+1)$, and that $b_1(i)$ is not all-ones. Then $b(i)$ has a zero in the last $m$ bits, so adding one leaves the first $m$ bits unchanged. Therefore we have $r \circ b_0(i+1) = r \circ b_0(i)$ and $b_1(i+1) = b_1(i) + 1$. Thus if $\psi(i) = (x, y)$, we have $\psi(i+1) = (x+1, y)$, producing an embedding $\psi(e)$ along a horizontal grid line between two lattice points.
        
    \textbf{Case 3 (Vertical outer edges):} Suppose $e = (i, S(i))$ is an edge in $C_\pi$, where $b_0(i)$ is all-ones but $b_1(i)$ is not. This corresponds to the top row of lattice points in the square, except for the top right most point. 
    Let $k$ be the first zero from the left in $b_1(i)$. The top side of $S$ has segment identifications by halving on the right, so our point $\psi(i)$ lies directly below the segment $h_k$ on top. 
    The start of segment $h_k$ in the top boundary is at $x = \delta + \sum_{j=1}^{k-1} 2^{m-j}$, while it starts in the bottom boundary at $\delta + 2^{m-k}$. Therefore, the vertical grid line $x = b_1(i)$ going up from $\psi(i)$ goes through the segment $h_k$ and turns into the grid line
    \[x = b_1(i) + \delta + 2^{m-k} - \delta - \sum_{j=1}^{k-1} 2^{m-j}.\]
    This corresponds to the binary string where the $k-1$ ones at the start of $b_1(i)$ are turned to zeros, while the zero at the $k$-th position turns to one, and all bits after are left unchanged. Therefore the next lattice point hit by this grid line is $(b_1(i) + 2^{m-k} - \sum_{j=1}^{k-1} 2^{m-j}, 0\dots 0) = \psi(S(i)).$ This produces an embedding $\psi(e)$ along a vertical grid line jumping across a segment identification of the surface.   
    
    \textbf{Case 4 (Horizontal outer edges):} 
    Suppose $e = (i, i+1)$ is an edge in $C_1$, where $b_1(i)$ is all-ones but $b_0(i)$ is not. This corresponds to the rightmost row of lattice points in the square, except for the top right most point. 
    Let $k$ be the first zero from the right in $b_0(i)$, which makes it the first zero from the left in $r \circ b_0(i)$. The right side of $S$ has segment identifications by halving on the top, so our point $\psi(i)$ lies directly left of the segment $v_k$ to the right. By a similar calculation as above, the horizontal grid line $x = r\circ b_0(i)$ going right from $\psi(i)$ goes through the segment $v_k$ and turns into the grid line
    \[y = r \circ b_0(i) + \delta + 2^{m-k} - \delta - \sum_{j=1}^{k-1} 2^{m-j}.\]
    This corresponds to the reverse of the binary string with the $k-1$ zeroes to the right, a one at the $k$-th position, and the same remaining bits to the left as $b_0(i)$. This is precisely $r(b_0(i) + 1)$, the $y$-coordinate of $\psi(i+1)$. Thus we have produced an embedding $\psi(e)$ along a horizontal grid line jumping across a segment identification of the surface. 
    
    \textbf{Case 5 (Edges requiring reroute):} The remaining edges are the two edges, one in $C_1$ and the other in $C_\pi$, starting from the point where both $b_0(i)$ and $b_1(i)$ are all-ones. Therefore $i = N-1$, for which both taking the successor and adding 1 modulo $N$ lead to $0$. Thus both our edges are of the form $(N-1, 0)$. We embed these edges into two curves connecting $\psi(0) = (0, 0)$ and $\psi(N-1) = (2^m - 1, 2^m - 1)$, passing through the segments $h_{m+1}$ and $v_{m+1}$ respectively.  
    
    The vertical and horizontal edges going down and left respectively from $\psi(0) = (0, 0)$ pass through the sides at a distance $-\delta = 1/2 + \varepsilon$ from the bottom left corner of the square. This means they pass through the segments $h_{m+1}$ and $v_{m+1}$ respectively. On the other hand, $\psi(N-1) = (2^m - 1, 2^m - 1)$ is a distance of $1/2 - \varepsilon$ away in both dimensions from the top right corner of the square. Since $\varepsilon$ is chosen to be small enough, the vertical and horizontal grid lines from $\psi(N-1)$ pass through the segments $h_{m+2}$ and $v_{m+2}$ respectively. We reroute these grid lines slightly to instead meet the segments $h_{m+1}$ and $v_{m+1}$ at the points where the grid lines from $(0, 0)$ emerge, as illustrated by the magenta edges in \autoref{fig:vdc-embedding}. These rerouted grid lines don't have any crossings, since there are no other edges near these corners. Therefore these curves are the desired embeddings of the edges $(N-1, 0)$.  
\end{proof}

\subsection{Smaller sequence graphs embed as minors} \label{sec:top-subgraphs}
For a fixed sequence $(a_n)$, let $G_N$ be the $N$-th sequence graph.
One might expect that as $N$ varies, the changes to the ordering $\pi$ lead to significant changes in the structure of the $C_\pi$ edges in $G_N$.
In fact, these changes are quite predictable, and are explained by the notion of minors from graph theory. 

A graph $H$ is said to be a \textbf{minor} of a graph $G$ if $H$ can be obtained from $G$ via a sequence of vertex deletions, edge deletions and edge contractions. It is well known that if a graph $G$ can be embedded into some surface $S$ then so can any minor $H$ of $G$ \changed{(one can check none of the three operations break the embedding); see for example \cite{mohar2001minors}} .

\begin{lemma}\label{thm:top-subgraphs}
    For $N<M$, let $G_N, G_M$ be the $N$-th and $M$-th sequence graphs respectively and suppose $G_N'$ is the graph obtained by deleting the $C_1$ edge $(N-1, 0)$ in $G_N$. Then $G'_N$ is a minor of $G_M$.
\end{lemma}

\begin{proof}
    We describe here the sequence of edge deletions and edge contractions that transforms $G_M$ into $G_N'$. First, consider the cycle $C_{1, M}$ in $G_M$. We can remove the edges $(N,N+1),\ldots,(M-2,M-1)$ and $(M-1,0)$. The leftover piece of the cycle $C_{1, M}$ contains the same edges as $C_{1, N}$ in $G_N$, except for the edge $(N-1, 0)$ which was removed to form $G'_N$. It remains to modify the other cycle.

     The vertices from $N$ to $M$ are now of degree 2, with those edges coming exclusively from $C_{\pi,M}$. Remember that $C_{\pi,M}$ was obtained from the ordering of the $a_0,\ldots,a_{M-1}$. Starting from $C_{\pi,N}$ and the first $N$ elements of the sequence $a_0,\ldots,a_{N-1}$, adding the next element $a_N$ can be done by subdividing one of the edges of the cycle and including $a_N$ ``in its position''. Repeating this process $M-N$ times gives $C_{\pi,M}$. Naturally, we can reverse this process via edge contractions: contract one of the two remaining edges adjacent to $M-1$, then to $M-2$, and so on until $N$. Since all these vertices are of degree 2, which edge is contracted has no impact on the final cycle. We are left with an $N$ vertex graph, and the two ``cycles'' are identical: $G_N'$ is a minor of $G_M$.
\end{proof}

More precisely, our proof actually shows that $G'_N$ is a \emph{topological minor} of $G_M$, since the edge contractions are all actually subdivisions. We have now shown that if the $M$-th sequence graph embeds into some surface, then for all $N<M$ the $N$-th sequence graph embeds into the same surface up to one edge. \autoref{thm:kron-hard} and \autoref{thm:vdc-hard} immediately follow.

\subsection{\changed{IET sequence graphs embed into associated translation surfaces}} \label{sec:iet-proof}
This section starts with an introduction to the theory of interval exchange transforms and their associated translation surfaces, before stating and proving \autoref{thm:iet-embedding-thm}. For a more detailed introduction to the theory, see the notes by Viana \cite{viana2006ergodic} or Yoccoz \cite{yoccoz2010interval}.

Given a permutation $\pi: [k] \to [k]$ and a choice of subinterval lengths $\lambda = (\lambda_1, \dots, \lambda_n)$ such that $\sum_{j=1}^k \lambda_j = 1$, the associated \textbf{interval exchange transformation} (IET) $T_{\pi, \lambda}: [0, 1] \to [0, 1]$ is the map which acts piecewise linearly on subintervals of $[0,1]$ with lengths in $\lambda$, rearranging them so the subinterval at position $j$ is moved to position $\pi(j)$.
More precisely, for $j \in [k]$ let 
\[
    s_j = \sum_{\ell=1}^{j-1} \lambda_\ell \quad \text{and} \quad s_j' = \sum_{\ell=1}^{\pi(j)-1} \lambda_{\pi^{-1}(\ell)}.
\]
Then if $x$ lies in the subinterval $[s_j, s_{j+1})$, we have $T_{\pi, \lambda}(x) =  x - s_j + s_j'$. The maps in \autoref{fig:vdc-kron-iets} are examples of IETs, as is the map on the left side of \autoref{fig:iet-surfaces}.

Both the Kronecker and the van der Corput sequence can be defined by setting $a_0 = 0$ and $a_{i+1} = T(a_i)$, where $T$ is an appropriate map. For the Kronecker sequence with parameter $0 < \theta < 1$, $T$ is the IET defined by rotation by $\theta$ around the unit circle, with $\pi = (2, 1)$ and $\lambda = (\theta, 1-\theta)$. The map for the van der Corput sequence is like an IET that involves countably many intervals, and is called the \emph{dyadic odometer} \cite{downaroicz2005survey, iommi2024odometers} or \emph{von Neumann–Kakutani adding machine} \cite{bruin2023odometers}. 

In general, any IET $T$ can define a sequence with $a_0 = 0$ and $a_{i+1} = T(a_i)$. If the orbit $(T^n(0))_{n \in \mathbb{Z}_{\ge 0}}$ of $0$ is dense under the map $T$, the produced sequences have sequence graphs which look like surfaces with a specific genus (up to deleting the single rogue edge $(N-1, 0)$). \autoref{fig:iet-surfaces} illustrates an example of this phenomenon. 
\begin{figure}[!ht]
    \centering
    \begin{minipage}{0.48\textwidth}
        \centering
        \includegraphics[width=0.75\textwidth]{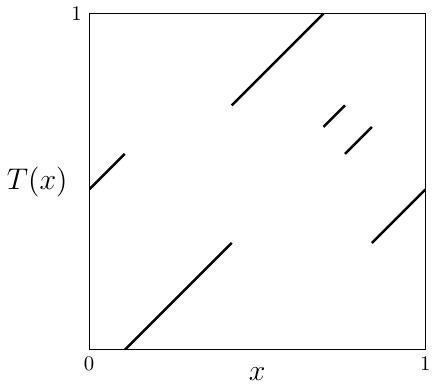}
    \end{minipage}
    \begin{minipage}{0.48\textwidth}
        \centering
        \includegraphics[width=0.75\textwidth]{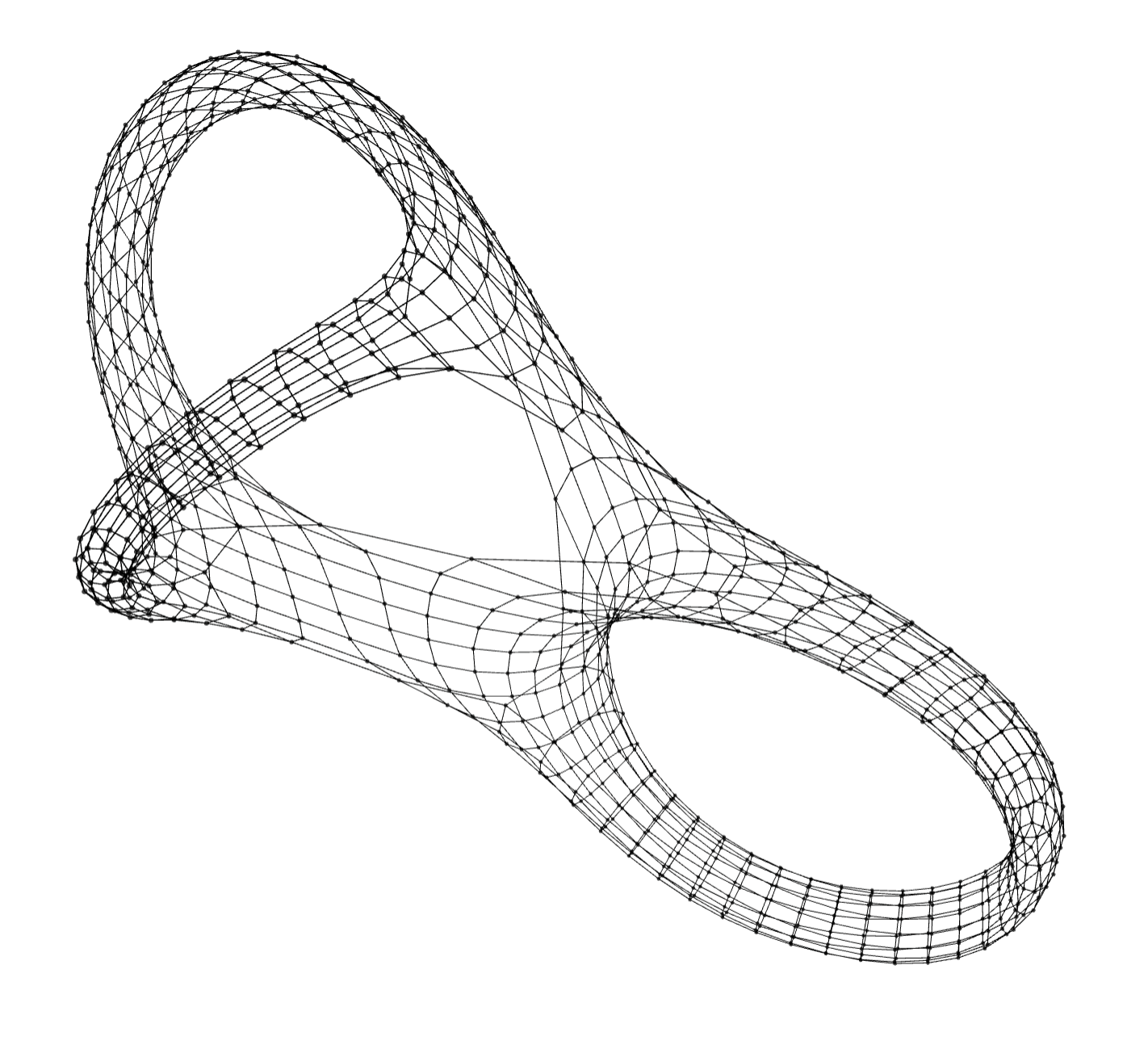}
    \end{minipage}
    \caption{The sequence generated by iterating the IET $T$ (left) produces a surface-like sequence graph with $N=1000$ (right).}
    \label{fig:iet-surfaces}
\end{figure} 

The assumption that $T$ is an IET such that the orbit of 0 is dense ensures that the produced sequence has ``good dynamics''. This condition would be implied by minimality (every point has dense orbit) or unique ergodicity of $T$. We can still obtain an embedding (up to deleting the rogue edge) for finite sequence graphs without this condition, but the resulting graph might be concentrated in one region of the surface or have vertices that accumulate onto a singularity. 

For every IET $T$, one can build a translation surface $S$ such that the discrete dynamic system of iteration by $T$ encodes the dynamics of a continuous straight-line flow on $S$ along a particular line. To be precise, we introduce the notion of \emph{first return maps} and \emph{horizontal transversals}. Since $S$ is a translation surface, we can choose a frame in one chart which extends by translations to an identical frame on all charts. This gives a well-defined `horizontal' and `vertical' direction. With this choice, a \emph{horizontal transversal} is a closed line segment along the horizontal direction. From any point along this transversal $H$, we can flow along the vertical direction, and observe where we next return to $H$. This defines a \emph{first return map} $f: H \to H$ which discretizes the dynamics of the continuous flow in the vertical direction. Given an IET $T$, we want to associate to it a translation surface $S$ with a horizontal transversal $H$ such that $T$ is exactly the first return map of the vertical flow on $S$ to $H$ (under the identification $H\cong [0, 1]$). Such a surface exists by Veech's construction \cite{veech1982IET}; see the notes by Yoccoz \cite[\S 4]{yoccoz2010interval} for a detailed introduction. 

\begin{figure}[!ht]
    \centering
    \begin{minipage}{0.4\textwidth}
        \centering
        \includegraphics[width=0.85\textwidth]{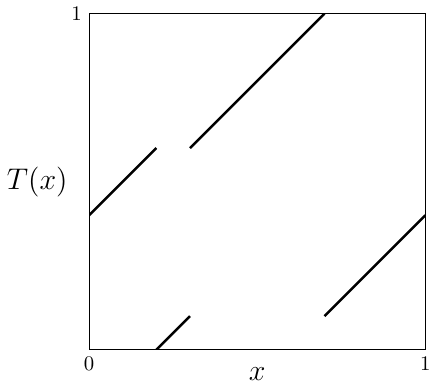}
    \end{minipage}
    \begin{minipage}{0.56\textwidth}
        \centering
        \includegraphics[width=0.85\textwidth]{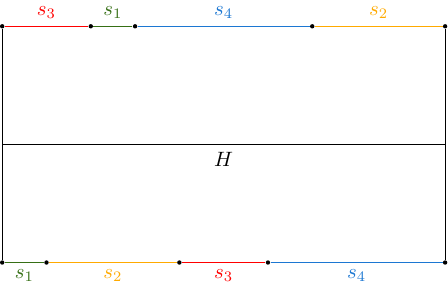}
    \end{minipage}
    \caption{The IET on the left is the first return map to transversal $H$ of the vertical flow on the translation surface on the right.}
    \label{fig:iet-example-surface}
\end{figure}

The IETs corresponding to rotation by $\theta$, whose iterations form the Kronecker sequence, are associated by this construction to the torus. Our main tool for proving the Kronecker sequence graph embedding in \autoref{thm:kron-nice} is the three gap theorem, which naturally generalizes to IETs as a $(d + 2)$-gap theorem \cite{taha2018gaptheorem}. The dyadic odometer, whose iterations form the van der Corput sequence, is associated by analogues of the above construction to the Chamanara surface \cite{lindsey2016flatsurfmodels}. This connection between the Chamanara surface and the dyadic odometer (and hence to the van der Corput sequence produced by iterating the map) highlights why it is the natural surface for embedding van der Corput sequence graphs into. The relationship between IET, translation surface and sequence graph extends to the general case.

\begin{theorem} \label{thm:iet-embedding-thm} 
    Suppose $T: [0, 1] \to [0, 1]$ is an IET such that the orbit of 0 is dense in the unit interval, and $S$ is a translation surface such that $T$ is the first return map of the vertical flow to a horizontal transversal in $S$. Then the sequence defined by $a_0 = 0$ and $a_{i+1} = T(a_i)$ has sequence graphs which embed into $S$ after deleting the rogue edge $(N-1, 0)$. There is an infinite subsequence $N_k$ such that the sequence graphs for $N = N_k+1$ embed into $S$ without deleting any edges.
\end{theorem}
\begin{proof}[Proof of \autoref{thm:iet-embedding-thm}]
    Suppose $H$ is the horizontal transversal specified in the theorem statement $S$, and $f$ the first return map of the vertical flow on $H$. For a generic point on $H$, the vertical flow from that point will never hit a singularity (we have finitely many singularities, and hence countably many points on $H$ that can hit a singularity). Embed the vertex 0 onto such a generic point; call it $\psi(0)$. Then for $i = 1, \dots, N-1$, we embed the vertex $i$ to the point $\psi(i) = f^i(\psi(0))$. 

    We first embed our $C_\pi$ edges along the transversal $H$. We know that $T$ agrees with the first return map $f$, and so the point $a_{\pi(i)} = T^{\pi(i)}(0)$ corresponds to the point $f^{\pi(i)}(\psi(0)) = \psi(\pi(i))$ on $H$. Therefore if we start at $\psi(0)$ and go along $H$, we will go through our embedded vertices in sorted order. We can therefore label the corresponding segments of $H$ as the embeddings of edges in $C_\pi$. Since $H$ is closed, this produces an embedding of the entire Hamiltonian cycle.

    Now we embed the $C_1$ edges. For any $i = 0, \dots, N-2$, we have $\psi(i+1) = f(\psi(i))$. There is a segment of the vertical flow starting at $\psi(i)$ and ending at $\psi(i+1)$, which does not intersect $H$. These segments do not intersect our embedded $C_\pi$ edges, and will be our embeddings for the edges $(i, i+1)$. Therefore we have an embedding of the entire graph, except for the edge $(N-1, 0)$ in $C_1$.

    Since the orbit of 0 is dense, there is an infinite subsequence $N_k$ such that $T^{N_k}(0) = a_{N_k}$ is larger than all terms in the sequence before it. Then in the $(N_k + 1)$-th sequence graph, we can start at $\psi(N_k)$ and reach $\psi(0)$ by going horizontally without intersecting any existing vertical flow segments. Fudging this path slightly, we also avoid intersecting $H$. This produces an embedding of the rogue edge $(N_k, 0)$ (which is $(N-1, 0)$ in this case as $N = N_k + 1$), and produces an infinite subsequence of sequence graphs that can be embedded without any edge deletions.
\end{proof}

\begin{figure}[!ht]
    \centering
    \begin{minipage}{0.48\textwidth}
        \centering
        \includegraphics[width=0.85\textwidth]{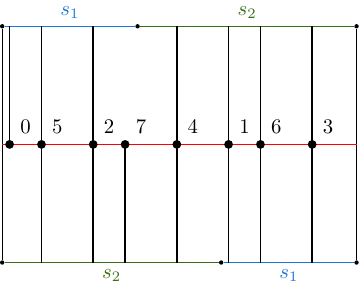}
    \end{minipage}
    \begin{minipage}{0.48\textwidth}
        \centering
        \includegraphics[width=0.85\textwidth]{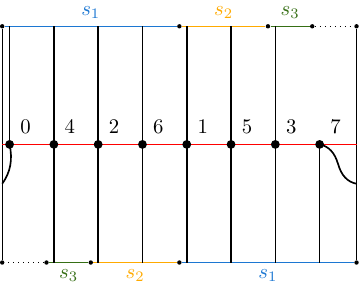}
    \end{minipage}
    \caption{Embeddings of the sequence graph following \autoref{thm:iet-embedding-thm} for Kronecker sequence (left) and van der Corput (right).}
    \label{fig:kron-vdc-iet-embeddings}
\end{figure}

\begin{example}
    We shall apply the construction for \autoref{thm:iet-embedding-thm} to \autoref{ex:kron-8-graph} and \autoref{ex:vdc-8-graph}. First consider the Kronecker sequence with parameter $\theta = 1.618\dots$. Following the vertical flow from a generic point, we note that the vertices on the horizontal transversal follow the sorted order $(0, 5, 2, 7, 4, 1, 6, 3)$, which produces the $C_\pi$ cycle edges. However, here the final sequence term $a_7$ is not the highest. Therefore we don't have an easy segment from the vertex $7$ to the vertex $0$ that does not cross other edge embeddings. This produces an embedding of the sequence graph after deleting the rogue edge $(7, 0)$ (\autoref{fig:kron-vdc-iet-embeddings}, left). However, for the binary van der Corput sequence graph with $N=8$, $a_7$ is the largest of all sequence elements. Therefore we can add a fudged edge to obtain an embedding of the entire sequence graph (\autoref{fig:kron-vdc-iet-embeddings}, right).
\end{example}

\subsection*{Acknowledgements} 
The authors thank Jayadev Athreya and Stefan Steinerberger for many helpful conversations. \changed{We also thank the anonymous referee for helpful comments and for the statement and main ideas for \autoref{thm:iet-embedding-thm}.}

\printbibliography

\end{document}